\documentclass[preprint,12pt]{elsarticle}
\usepackage{amsmath}
\usepackage{amsfonts}
\usepackage{amssymb}
\usepackage{graphicx}%
\providecommand{\U}[1]{\protect\rule{.1in}{.1in}}
\newtheorem{theorem}{Theorem}
\newtheorem{acknowledgement}[theorem]{Acknowledgement}

\newtheorem{example}[theorem]{Example}

\newtheorem{lemma}[theorem]{Lemma}

\newtheorem{remark}[theorem]{Remark}

\newenvironment{proof}[1][Proof]{\noindent\textbf{#1.} }{\ \rule{0.5em}{0.5em}}

\journal{Journal of Computational and Applied Mathematics}

\begin{document}

\begin{frontmatter}

\title{On Taylor series and Kapteyn series of the first and second type}
\author{Diego Ernesto Dominici \fnref{PA}}

\address{Technische Universit\"at Berlin, Stra\ss e des 17.~Juni 136,\\D-10623 Berlin, Germany}

\fntext[PA]{Permanent address: Department of Mathematics State
University of New York at New Paltz, 1~Hawk Dr., New Paltz, NY 12561-2443,
USA}

\begin{abstract}
We study the relation between the coefficients of Taylor series and Kapteyn
series representing the same function. We compute explicit formulas for
expressing one in terms of the other and give examples to illustrate our method.

\end{abstract}

\begin{keyword}
Kapteyn series \sep Taylor series \sep Bessel functions
\MSC 42C10 \sep 30B10 \sep 33C10
\end{keyword}

\end{frontmatter}

\section{Introduction}

Series of the form \cite{MR1349110}
\begin{equation}%
{\displaystyle\sum\limits_{n=0}^{\infty}}
\alpha_{n}^{\nu}\mathrm{J}_{n+\nu}\left[  \left(  n+\nu\right)  z\right]  ,
\label{Kfirst}%
\end{equation}
and%
\begin{equation}%
{\displaystyle\sum\limits_{n=0}^{\infty}}
\alpha_{n}^{\mu,\nu}\mathrm{J}_{\mu+n}\left[  \left(  \mu+\nu+2n\right)
z\right]  \mathrm{J}_{\nu+n}\left[  \left(  \mu+\nu+2n\right)  z\right]  ,
\label{Ksecond}%
\end{equation}
where $\mu,\nu\in\mathbb{C}$ and \textrm{$J$}$_{n}\left(  \cdot\right)  $ is
the Bessel function of the first kind, are called \textit{Kapteyn series of
the first kind }and \textit{Kapteyn series of the second kind }respectively.

Kapteyn series have a long history, going back to Joseph Louis de Lagrange's
1771 paper \emph{Sur le Probl\`{e}me de K\'{e}pler }\cite{Lagrange1771}, where
he solved Kepler's equation \cite{MR1268639}%
\begin{equation}
M=E-\varepsilon\sin\left(  E\right)  , \label{KE}%
\end{equation}
using his method for solving implicit equations \cite{Lagrange1768} (now
called \textit{Lagrange inversion theorem}) and obtained \cite{MR1346357}
\[
E(M)=M+%
{\displaystyle\sum\limits_{n=1}^{\infty}}
\frac{\varepsilon^{n}}{n!}\frac{d^{n-1}}{dM^{n-1}}\sin^{n}\left(  M\right)  .
\]
Here $M$ is the mean anomaly (a parameterization of time) and $E$ is the
eccentric anomaly (an angular parameter) of a body orbiting on an ellipse with
eccentricity $\varepsilon$.

In 1819 Friedrich Wilhelm Bessel published his paper \emph{Analytische
Aufl\"{o}sung der Kepler'schen Aufgabe} \cite{Bessel}, where he approached
(\ref{KE}) using a different method. First of all he observed that the
function $g(M)=E(M)-M$ defined implicitly by $g=\varepsilon\sin\left(
g+M\right)  $ is $2\pi-$periodic and satisfies $g(0)=0=g(\pi).$ Hence, $g(M)$
can be expanded in a Fourier sine series
\[
g(M)=%
{\displaystyle\sum\limits_{n=1}^{\infty}}
b_{n}\sin\left(  nM\right)  ,
\]
where
\begin{align*}
b_{n}  &  =\frac{2}{\pi}%
{\displaystyle\int\limits_{0}^{\pi}}
g(M)\sin(nM)dM=\\
&  =-\frac{2}{\pi}\left[  g(M)\frac{\cos(nM)}{n}\right]  _{0}^{\pi}+\frac
{2}{\pi n}%
{\displaystyle\int\limits_{0}^{\pi}}
\cos(nM)dg\\
&  =\frac{2}{\pi n}%
{\displaystyle\int\limits_{0}^{\pi}}
\cos(nM)d\left(  E-M\right) \\
&  =\frac{2}{\pi n}%
{\displaystyle\int\limits_{0}^{\pi}}
\cos\left[  n\left(  E-\varepsilon\sin E\right)  \right]  dE-\frac{2}{\pi n}%
{\displaystyle\int\limits_{0}^{\pi}}
\cos(nM)dM
\end{align*}
and hence%
\[
b_{n}=\frac{2}{\pi n}%
{\displaystyle\int\limits_{0}^{\pi}}
\cos\left(  nE-n\varepsilon\sin E\right)  dE.
\]
He then introduced the function \textrm{$J$}$_{n}(z)$ defined by%
\begin{equation}
\text{\ }\mathrm{J}_{n}(z)=\frac{1}{\pi}%
{\displaystyle\int\limits_{0}^{\pi}}
\cos\left(  nE-z\sin E\right)  dE,\quad n\in\mathbb{Z} \label{Jint}%
\end{equation}
which now bears his name and obtained%
\begin{equation}
E(M)=M+%
{\displaystyle\sum\limits_{n=1}^{\infty}}
\frac{2}{n}\mathrm{J}_{n}(n\varepsilon)\sin\left(  nM\right)  . \label{E(M)}%
\end{equation}
Bessel's work on (\ref{Jint}) was continued by other researchers including
Lommel, who defined the Bessel function of the first kind by \cite{Lommel}%
\begin{equation}
\mathrm{J}_{\nu}\left(  z\right)  =%
{\displaystyle\sum\limits_{n=0}^{\infty}}
\frac{\left(  -1\right)  ^{n}}{n!\Gamma\left(  \nu+n+1\right)  }\left(
\frac{z}{2}\right)  ^{\nu+2n},\quad\nu\in\mathbb{C}, \label{bessel}%
\end{equation}
where $\Gamma(\cdot)$ is the Gamma function.

In 1817, Francesco Carlini \cite{Carlini} found an expression for the true
anomaly $v$ (an angular parameter), defined in terms of $E$ and $\varepsilon$
by%
\[
\tan\left(  \frac{v}{2}\right)  =\sqrt{\frac{1+\varepsilon}{1-\varepsilon}%
}\tan\left(  \frac{E}{2}\right)  .
\]
Carlini's expression reads \cite{MR1140277}
\[
v=M+%
{\displaystyle\sum\limits_{n=1}^{\infty}}
B_{n}\sin\left(  nM\right)  ,
\]
where%
\[
B_{n}=\frac{2}{n}\mathrm{J}_{n}(n\varepsilon)+%
{\displaystyle\sum\limits_{m=0}^{\infty}}
\alpha^{m}\left[  \mathrm{J}_{n-m}(n\varepsilon)+\mathrm{J}_{n+m}%
(n\varepsilon)\right]  ,
\]
with $\varepsilon=\frac{2\alpha}{1+\alpha^{2}}.$ The problem considered by
Carlini was to determine the asymptotic behavior of the coefficients $B_{n}$
for large values of $n$ \cite{MR1915514}. The astronomer Johann Franz Encke
drew Carl Gustav Jacob Jacobi's attention to the work of Carlini. In 1849,
Jacobi published a paper improving and correcting Carlini's article
\cite{Jacobi1849} and in 1850 Jacobi published a translation from Italian into
German \cite{Jacobi1850}, with critical comments and extensions of Carlini's investigation.

Bessel's research on series of the type (\ref{E(M)}) was continued by Ernst
Meissel \cite{MR1344309} in his papers \cite{Meissel1}, \cite{Meissel2} and by
Willem Kapteyn (not to be confused with his brother Jacobus Cornelius Kapteyn
\cite{Kapteyn}) in the articles \cite{MR1508887} and \cite{Kapteyn1}. Most of
the early work on Kapteyn series, together with their own results, can be
found in the books by Niels Nielsen \cite[Chapter XXII]{Nielsen} and George
Neville Watson \cite[Chapter 17]{MR1349110}.

In recent years, there has been a renewed interest on Kapteyn series,
particularly from researchers in the fields of Astrophysics and
Electrodynamics (see \cite{Tautz2} for a review of current applications). In
\cite{Tautz1}, Ian Lerche and Robert C. Tautz studied the Kapteyn series of
the second kind%
\[
S_{1}(a)=%
{\displaystyle\sum\limits_{n=1}^{\infty}}
n^{4}\mathrm{J}_{n}^{2}\left(  na\right)
\]
and derived the formula%
\[
S_{1}(a)=\frac{a^{2}\left(  64+592a^{2}+472a^{4}+27a^{6}\right)  }{256\left(
1-a^{2}\right)  ^{\frac{13}{2}}}.
\]
They continued their investigations in \cite{MR2449307}, where they outlined a
way for calculating more general Kapteyn series of the form%
\begin{equation}
S_{1}(m,a)=%
{\displaystyle\sum\limits_{n=1}^{\infty}}
n^{2m}\mathrm{J}_{n}^{2}\left(  na\right)  ,\quad m=0,1,\ldots. \label{S1}%
\end{equation}

The purpose of this paper is to describe a method for computing the
coefficients in the Taylor series of functions defined by Kapteyn series of
the first (\ref{Kfirst}) and second (\ref{Ksecond}) kind. As an example, we
will show a closed-form expression of (\ref{S1}) valid for all values of $m.$

\section{Kapteyn series of the first kind}

We begin by considering functions expressed as Kapteyn series of the first kind.

\begin{theorem}
Suppose that%
\begin{equation}
f\left(  z\right)  =%
{\displaystyle\sum\limits_{m=0}^{\infty}}
b_{m}z^{m} \label{f1}%
\end{equation}
and%
\begin{equation}
z^{\nu}f\left(  z\right)  =%
{\displaystyle\sum\limits_{n=0}^{\infty}}
a_{n}^{\nu}\mathrm{J}_{\nu+n}\left[  \left(  \nu+n\right)  z\right]  ,
\label{f2}%
\end{equation}
where both series converge absolutely for $z$ in some domain $\Omega.$ Then,
we have%
\begin{equation}
a_{s}^{\nu}=%
{\displaystyle\sum\limits_{m=0}^{\left\lfloor \frac{s}{2}\right\rfloor }}
v_{s,m}b_{s-2m} \label{an1}%
\end{equation}
and%
\begin{equation}
b_{s}=%
{\displaystyle\sum\limits_{m=0}^{\left\lfloor \frac{s}{2}\right\rfloor }}
u_{s,m}a_{s-2m}^{\nu} \label{bn1}%
\end{equation}
for all values of $\nu,$ with%
\[
u_{n,k}=\frac{\left(  -1\right)  ^{k}}{k!\Gamma\left(  \nu+n-k+1\right)
}\left(  \frac{\nu+n-2k}{2}\right)  ^{\nu+n}%
\]
and%
\[
v_{n,k}=\frac{1}{2}\frac{\left(  \nu+n-2k\right)  ^{2}\Gamma\left(
\nu+n-k\right)  }{k!}\left(  \frac{2}{\nu+n}\right)  ^{\nu+n-2k+1}.
\]

\end{theorem}

\begin{proof}
Let $z\in\Omega.$ To prove (\ref{bn1}), we use (\ref{f1}) and (\ref{f2}), to
get%
\[%
{\displaystyle\sum\limits_{m=0}^{\infty}}
b_{m}z^{\nu+m}=z^{\nu}f\left(  z\right)  =%
{\displaystyle\sum\limits_{n=0}^{\infty}}
a_{n}^{\nu}\mathrm{J}_{\nu+n}\left[  \left(  \nu+n\right)  z\right]
\]
and from (\ref{bessel}) we obtain%
\[%
{\displaystyle\sum\limits_{s=0}^{\infty}}
b_{s}z^{s+\nu}=%
{\displaystyle\sum\limits_{n=0}^{\infty}}
a_{n}^{\nu}%
{\displaystyle\sum\limits_{m=0}^{\infty}}
\frac{\left(  -1\right)  ^{m}}{m!\Gamma\left(  \nu+n+m+1\right)  }\left[
\frac{\left(  \nu+n\right)  z}{2}\right]  ^{\nu+n+2m}%
\]
or%
\begin{equation}%
{\displaystyle\sum\limits_{s=0}^{\infty}}
b_{s}z^{s}=%
{\displaystyle\sum\limits_{n=0}^{\infty}}
a_{n}^{\nu}%
{\displaystyle\sum\limits_{m=0}^{\infty}}
\frac{\left(  -1\right)  ^{m}}{m!\Gamma\left(  \nu+n+m+1\right)  }\left(
\frac{\nu+n}{2}\right)  ^{\nu+n+2m}z^{n+2m}. \label{eq1}%
\end{equation}
Setting $n+2m=s$ on the right-hand side of (\ref{eq1}), we have%
\[%
{\displaystyle\sum\limits_{s=0}^{\infty}}
b_{s}z^{s}=%
{\displaystyle\sum\limits_{s=0}^{\infty}}
z^{s}%
{\displaystyle\sum\limits_{m=0}^{\left\lfloor \frac{s}{2}\right\rfloor }}
\frac{\left(  -1\right)  ^{m}a_{s-2m}^{\nu}}{m!\Gamma\left(  \nu+s-m+1\right)
}\left(  \frac{\nu+s-2m}{2}\right)  ^{\nu+s},
\]
from which (\ref{bn1}) follows.

We have%
\begin{gather*}%
{\displaystyle\sum\limits_{j=k}^{p}}
u_{2j,j-k}v_{2s,s-j}=\frac{\left(  -1\right)  ^{p+k}2^{\nu+2p}}{\Gamma\left(
s-p\right)  }\left(  \frac{\nu+2k}{\nu+2s}\right)  ^{\nu+2p+1}\\
\times\frac{\left(  \nu+2k\right)  }{\left(  s-k\right)  \left(
\nu+s+k\right)  }\binom{\nu+p+s}{p-k}%
\end{gather*}
and hence%
\[%
{\displaystyle\sum\limits_{j=k}^{s}}
u_{2j,j-k}v_{2s,s-j}=0,\quad k\neq s.
\]
When $k=s,$ we get%
\begin{gather*}
u_{2s,0}v_{2s,0}=\frac{1}{\Gamma\left(  \nu+2s+1\right)  }\left(  \frac
{\nu+2s}{2}\right)  ^{\nu+2s}\\
\times\left(  \nu+2s\right)  \Gamma\left(  \nu+2s\right)  \left(  \frac{2}%
{\nu+2s}\right)  ^{\nu+2s}=1
\end{gather*}
and therefore%
\[%
{\displaystyle\sum\limits_{j=k}^{s}}
u_{2j,j-k}v_{2s,s-j}=\delta_{k,s}.
\]
Thus,%
\begin{gather*}%
{\displaystyle\sum\limits_{m=0}^{s}}
v_{2s,m}b_{2\left(  s-m\right)  }=%
{\displaystyle\sum\limits_{j=0}^{s}}
v_{2s,s-j}b_{2j}=%
{\displaystyle\sum\limits_{j=0}^{s}}
v_{2s,s-j}%
{\displaystyle\sum\limits_{m=0}^{j}}
u_{2j,m}a_{2\left(  j-m\right)  }^{\nu}\\
=%
{\displaystyle\sum\limits_{j=0}^{s}}
v_{2s,s-j}%
{\displaystyle\sum\limits_{k=0}^{j}}
u_{2j,j-k}a_{2k}^{\nu}=%
{\displaystyle\sum\limits_{k=0}^{s}}
a_{2k}^{\nu}%
{\displaystyle\sum\limits_{j=k}^{s}}
u_{2j,j-k}v_{2s,s-j}=\\%
{\displaystyle\sum\limits_{k=0}^{s}}
a_{2k}^{\nu}\delta_{k,s}=a_{2s}^{\nu}%
\end{gather*}
and a similar computation holds for $b_{2k+1}$ and $a_{2k+1}^{\nu},$ proving
(\ref{an1}).
\end{proof}

\begin{remark}
Formula (\ref{an1}) appeared in \cite[17.5 (6)]{MR1349110}, but it contains a
small mistake because there is a factor of $\frac{1}{2}$ missing in the
denominator. The result was also published in \cite[7.10.2 (29)]{MR698780},
but there is also a misprint there, since the factor $b_{n-2m}$ is missing in
the sum.

In \cite{MR2326079}, we computed formulas for $b_{s}$ and $a_{n}^{\nu}$ when
$\nu=0.$
\end{remark}

In order to find the coefficients $b_{s}$ for a particular choice of $\nu=0$
and $a_{n}^{\nu},$ we need the following result.

\begin{lemma}
Let $r,m\in N.$ Then, we have%
\begin{equation}%
{\displaystyle\sum\limits_{k=0}^{r}}
\binom{r}{k}\left(  -1\right)  ^{k}\left(  r-2k\right)  ^{m}=2^{r}\left.
\frac{d^{m}\sinh^{r}(t)}{dt^{m}}\right\vert _{t=0}. \label{binomial}%
\end{equation}

\end{lemma}

\begin{proof}
We have%
\[
\sinh^{r}(t)=\left(  \frac{e^{t}-e^{-t}}{2}\right)  ^{r}=\frac{1}{2^{r}}%
{\displaystyle\sum\limits_{k=0}^{r}}
\binom{r}{k}\left(  -1\right)  ^{k}e^{\left(  r-2k\right)  t}.
\]
Since%
\[
\frac{d^{m}e^{at}}{dt^{m}}=a^{m}e^{at},
\]
we obtain,%
\[
\left.  \frac{d^{m}\sinh^{r}(t)}{dt^{m}}\right\vert _{t=0}=\frac{1}{2^{r}}%
{\displaystyle\sum\limits_{k=0}^{r}}
\binom{r}{k}\left(  -1\right)  ^{k}\left(  r-2k\right)  ^{m}.
\]

\end{proof}

\begin{example}
Let's consider the special case $\nu=0$ and $a_{n}^{\nu}=n^{2p}.$ Then,
(\ref{bn1}) gives%
\begin{equation}
b_{s}(p)=\frac{1}{s!2^{s}}%
{\displaystyle\sum\limits_{k=0}^{\left\lfloor \frac{s}{2}\right\rfloor }}
\left(  -1\right)  ^{k}\binom{s}{k}\left(  s-2k\right)  ^{s+2p}. \label{bn2}%
\end{equation}
Since the terms in the sum are symmetric with respect to $k=\left\lfloor
\frac{s}{2}\right\rfloor ,$ we can write%
\[
b_{s}(p)=\frac{1}{s!2^{s}}\frac{1}{2}%
{\displaystyle\sum\limits_{k=0}^{s}}
\left(  -1\right)  ^{k}\binom{s}{k}\left(  s-2k\right)  ^{s+2p},
\]
unless $s=0=p,$ in which case we have%
\[
b_{0}(0)=1.
\]
Thus, we have%
\[
b_{s}(p)=\varepsilon_{s,p}\frac{1}{s!2^{s}}%
{\displaystyle\sum\limits_{k=0}^{s}}
\left(  -1\right)  ^{k}\binom{s}{k}\left(  s-2k\right)  ^{s+2p},
\]
with%
\begin{equation}
\varepsilon_{s,p}=\left\{
\begin{array}
[c]{c}%
1,\quad s=0=p\\
\frac{1}{2},\quad\text{otherwise}%
\end{array}
\right.  . \label{epsilon}%
\end{equation}
From (\ref{binomial}), we obtain%
\begin{align*}
b_{s}(p)  &  =\frac{\varepsilon_{s,p}}{s!}\left.  \frac{d^{s+2p}\sinh^{s}%
(t)}{dt^{s+2p}}\right\vert _{t=0}\\
&  =\frac{\varepsilon_{s,p}}{s!}\left(  s+2p\right)  !\left[  t^{s+2p}\right]
\sinh^{s}(t)
\end{align*}
or
\begin{equation}
b_{s}(p)=\varepsilon_{s,p}\left(  s+1\right)  _{2p}\left[  t^{2p}\right]
\left[  \frac{\sinh(t)}{t}\right]  ^{s}, \label{bs01}%
\end{equation}
where $\left[  t^{r}\right]  G(t)$ denotes the coefficient of $t^{r}$ in the
McLaurin series of $G(t).$ From (\ref{bs01}) we get
\begin{align*}
b_{s}(0)  &  =\varepsilon_{s,0},\quad b_{s}(1)=\frac{1}{2}\left(  s+1\right)
_{2}\frac{s}{6},\quad\\
b_{s}(2)  &  =\frac{1}{2}\left(  s+1\right)  _{4}\frac{s\left(  5s-2\right)
}{360},\quad\\
b_{s}(3)  &  =\frac{1}{2}\left(  s+1\right)  _{6}\frac{s\left(  35s^{2}%
-42s+16\right)  }{45360},\\
b_{s}(4)  &  =\frac{1}{2}\left(  s+1\right)  _{8}\frac{s\left(  5s-4\right)
\left(  35s^{2}-56s+36\right)  }{5443200},\ldots.
\end{align*}

If we define%
\[%
{\displaystyle\sum\limits_{n=0}^{\infty}}
n^{2p}\mathrm{J}_{n}\left(  nz\right)  =f_{p}\left(  z\right)  =%
{\displaystyle\sum\limits_{s=0}^{\infty}}
b_{s}(p)z^{s},
\]
it follows from (\ref{bs01}) that%
\[
f_{p}\left(  z\right)  =\left[  t^{2p}\right]
{\displaystyle\sum\limits_{s=0}^{\infty}}
\varepsilon_{s,p}\left(  s+1\right)  _{2p}\left[  \frac{\sinh(t)}{t}z\right]
^{s}%
\]%
\begin{equation}
f_{p}\left(  z\right)  =\left[  t^{2p}\right]  \left\{  \frac{1}{2}%
+\frac{\left(  2p\right)  !}{2}\left[  1-\frac{\sinh(t)}{t}z\right]
^{-(2p+1)}\right\}  . \label{fp}%
\end{equation}
Using (\ref{fp}), we get%
\begin{align*}
f_{0}\left(  z\right)   &  =\frac{2-z}{2\left(  1-z\right)  },\quad
f_{1}\left(  z\right)  =\frac{z}{2\left(  1-z\right)  ^{4}},\\
f_{2}\left(  z\right)   &  =\frac{z\left(  9z+1\right)  }{2\left(  1-z\right)
^{7}},\quad f_{3}\left(  z\right)  =\frac{z\left(  255z^{2}+54z+1\right)
}{2\left(  1-z\right)  ^{10}},\\
f_{4}\left(  z\right)   &  =\frac{z\left(  11025z^{3}+4131z^{2}+243z+1\right)
}{2\left(  1-z\right)  ^{13}},\ldots
\end{align*}
and in general%
\[
f_{n}\left(  z\right)  =\frac{z}{2}\frac{P_{n-1}(z)}{\left(  1-z\right)
^{3n+1}},
\]
where $P_{0}(z)=1$ and $P_{n}(z)$ is a polynomial of degree $n$ for
$n=1,2,\ldots.$

Using the relation \cite[17.33]{MR1349110}%
\[
f_{n+1}\left(  z\right)  =\frac{1}{1-z^{2}}\left(  z\frac{d}{dz}\right)
^{2}f_{n}\left(  z\right)  ,
\]
we find that the polynomials $P_{n}(z)$ satisfy
\begin{align*}
P_{n+1}(z)  &  =\frac{z^{2}\left(  z-1\right)  ^{2}}{z+1}P_{n}^{\prime\prime
}(z)+\frac{z}{z+1}\left[  \left(  1-6n\right)  z^{2}+(6n-4)z+3\right]
P_{n}^{\prime}(z)\\
&  +\frac{1}{z+1}\left[  9n^{2}z^{2}+(9n+1)z+1\right]  P_{n}(z).
\end{align*}

\end{example}

\section{Kapteyn series of the second kind}

We now consider functions expressed as Kapteyn series of the second kind.

\begin{theorem}
Suppose that%
\begin{equation}
f\left(  z\right)  =%
{\displaystyle\sum\limits_{m=0}^{\infty}}
b_{m}z^{m} \label{fb}%
\end{equation}
and%
\begin{equation}
z^{\mu+\nu}f\left(  z\right)  =%
{\displaystyle\sum\limits_{n=0}^{\infty}}
\left(  a_{n}^{\mu,\nu}+zc_{n}^{\mu,\nu}\right)  \mathrm{J}_{\mu+n}\left[
\left(  \mu+\nu+2n\right)  z\right]  \mathrm{J}_{\nu+n}\left[  \left(  \mu
+\nu+2n\right)  z\right]  , \label{kap2}%
\end{equation}
where both series converge absolutely for $z$ in some domain $\Omega.$ Then,
we have%
\begin{equation}
a_{s}^{\mu,\nu}=%
{\displaystyle\sum\limits_{k=0}^{s}}
\alpha_{s,k}^{\mu,\nu}b_{2k},\quad c_{s}^{\mu,\nu}=%
{\displaystyle\sum\limits_{k=0}^{s}}
\alpha_{s,k}^{\mu,\nu}b_{2k+1} \label{as}%
\end{equation}
and
\begin{equation}
b_{2s}=%
{\displaystyle\sum\limits_{k=0}^{s}}
\beta_{s,k}^{\mu,\nu}a_{k}^{\mu,\nu},\quad b_{2s+1}=%
{\displaystyle\sum\limits_{k=0}^{s}}
\beta_{s,k}^{\mu,\nu}c_{k}^{\mu,\nu}, \label{b2s}%
\end{equation}
for all $\mu,\nu,$ with%
\begin{align*}
\alpha_{s,k}^{\mu,\nu}  &  =\frac{\left(  \mu+\nu+2k\right)  ^{2}\Gamma\left(
\mu+k+1\right)  \Gamma\left(  \nu+k+1\right)  }{\left(  \mu+\nu+s+k\right)
\left(  \mu+\nu+2s\right)  }\\
&  \times\binom{\mu+\nu+s+k}{s-k}\left(  \frac{2}{\mu+\nu+2s}\right)
^{\mu+\nu+2k}%
\end{align*}
and%
\[
\beta_{s,k}^{\mu,\nu}=\frac{\left(  -1\right)  ^{s+k}}{\Gamma\left(
\mu+s+1\right)  \Gamma\left(  \nu+s+1\right)  }\binom{\mu+\nu+2s}{s-k}\left(
\frac{\mu+\nu+2k}{2}\right)  ^{\mu+\nu+2s}.
\]

\end{theorem}

\begin{proof}
Let $z\in\Omega.$ Using the formula \cite[5.41]{MR1349110}%
\[
\mathrm{J}_{\mu}\left(  z\right)  \mathrm{J}_{\nu}\left(  z\right)
=\sum_{k=0}^{\infty}\frac{\left(  -1\right)  ^{k}}{\Gamma\left(
\mu+k+1\right)  \Gamma\left(  \nu+k+1\right)  }\binom{\mu+\nu+2k}{k}\left(
\frac{z}{2}\right)  ^{\mu+\nu+2k}%
\]
in (\ref{kap2}), we get%
\begin{gather*}
z^{\mu+\nu}f\left(  z\right)  =%
{\displaystyle\sum\limits_{n=0}^{\infty}}
\left(  a_{n}^{\mu,\nu}+zc_{n}^{\mu,\nu}\right)  \sum_{k=0}^{\infty}%
\frac{\left(  -1\right)  ^{k}}{\Gamma\left(  \mu+n+k+1\right)  \Gamma\left(
\nu+n+k+1\right)  }\\
\times\binom{\mu+\nu+2n+2k}{k}\left[  \frac{\left(  \mu+\nu+2n\right)  z}%
{2}\right]  ^{\mu+\nu+2n+2k},
\end{gather*}
or%
\begin{align*}
f\left(  z\right)   &  =%
{\displaystyle\sum\limits_{n=0}^{\infty}}
\sum_{k=0}^{\infty}\left(  a_{n}^{\mu,\nu}+zc_{n}^{\mu,\nu}\right)
\frac{\left(  -1\right)  ^{k}}{\Gamma\left(  \mu+n+k+1\right)  \Gamma\left(
\nu+n+k+1\right)  }\\
&  \times\binom{\mu+\nu+2n+2k}{k}\left(  \frac{\mu+\nu+2n}{2}\right)
^{\mu+\nu+2n+2k}z^{2\left(  n+k\right)  }.
\end{align*}
Setting $n=s-k,$ we have%
\begin{align*}
f\left(  z\right)   &  =%
{\displaystyle\sum\limits_{s=0}^{\infty}}
\sum_{k=0}^{s}\left(  a_{s-k}^{\mu,\nu}+zc_{s-k}^{\mu,\nu}\right)
\frac{\left(  -1\right)  ^{k}}{\Gamma\left(  \mu+s+1\right)  \Gamma\left(
\nu+s+1\right)  }\\
&  \times\binom{\mu+\nu+2s}{k}\left(  \frac{\mu+\nu+2s-2k}{2}\right)
^{\mu+\nu+2s}z^{2s},
\end{align*}
or%
\begin{align*}
f\left(  z\right)   &  =%
{\displaystyle\sum\limits_{s=0}^{\infty}}
\frac{z^{2s}}{\Gamma\left(  \mu+s+1\right)  \Gamma\left(  \nu+s+1\right)
}\sum_{k=0}^{s}\left(  -1\right)  ^{k}\binom{\mu+\nu+2s}{k}\left(  \frac
{\mu+\nu+2s-2k}{2}\right)  ^{\mu+\nu+2s}a_{s-k}^{\mu,\nu}\\
&  +%
{\displaystyle\sum\limits_{s=0}^{\infty}}
\frac{z^{2s+1}}{\Gamma\left(  \mu+s+1\right)  \Gamma\left(  \nu+s+1\right)
}\sum_{k=0}^{s}\left(  -1\right)  ^{k}\binom{\mu+\nu+2s}{k}\left(  \frac
{\mu+\nu+2s-2k}{2}\right)  ^{\mu+\nu+2s}c_{s-k}^{\mu,\nu}.
\end{align*}
Comparing with (\ref{fb}), we conclude that%
\begin{align*}
b_{2s}  &  =\frac{1}{\Gamma\left(  \mu+s+1\right)  \Gamma\left(
\nu+s+1\right)  }\sum_{k=0}^{s}\left(  -1\right)  ^{k}\binom{\mu+\nu+2s}%
{k}\left(  \frac{\mu+\nu+2s-2k}{2}\right)  ^{\mu+\nu+2s}a_{s-k}^{\mu,\nu},\\
b_{2s+1}  &  =\frac{1}{\Gamma\left(  \mu+s+1\right)  \Gamma\left(
\nu+s+1\right)  }\sum_{k=0}^{s}\left(  -1\right)  ^{k}\binom{\mu+\nu+2s}%
{k}\left(  \frac{\mu+\nu+2s-2k}{2}\right)  ^{\mu+\nu+2s}c_{s-k}^{\mu,\nu},
\end{align*}
from which (\ref{b2s}) follows.

Since%
\begin{gather*}%
{\displaystyle\sum\limits_{k=j}^{p}}
\alpha_{s,k}^{\mu,\nu}\beta_{k,j}^{\mu,\nu}=\left(  -1\right)  ^{j+p}%
\binom{\mu+\nu+s+p}{s-p}\binom{\mu+\nu+2p}{p-j}\\
\times\left(  \frac{\mu+\nu+2j}{\mu+\nu+2s}\right)  ^{\mu+\nu+2p+2}%
\frac{\left(  p-s\right)  \left(  \mu+\nu+2s\right)  }{\left(  j-s\right)
\left(  \mu+\nu+s+j\right)  },
\end{gather*}
we see that%
\[%
{\displaystyle\sum\limits_{k=j}^{s}}
\alpha_{s,k}^{\mu,\nu}\beta_{k,j}^{\mu,\nu}=0,\quad j\neq s,
\]
while for $j=s,$ we have%
\begin{gather*}
\alpha_{s,s}^{\mu,\nu}\beta_{s,s}^{\mu,\nu}=\Gamma\left(  \mu+s+1\right)
\Gamma\left(  \nu+s+1\right)  \left(  \frac{2}{\mu+\nu+2s}\right)  ^{\mu
+\nu+2s}\\
\times\frac{1}{\Gamma\left(  \mu+s+1\right)  \Gamma\left(  \nu+s+1\right)
}\left(  \frac{\mu+\nu+2s}{2}\right)  ^{\mu+\nu+2s}=1.
\end{gather*}
Therefore,%
\begin{gather*}%
{\displaystyle\sum\limits_{k=0}^{s}}
\alpha_{s,k}^{\mu,\nu}b_{2k}=%
{\displaystyle\sum\limits_{k=0}^{s}}
\alpha_{s,k}^{\mu,\nu}%
{\displaystyle\sum\limits_{j=0}^{k}}
\beta_{k,j}^{\mu,\nu}a_{j}^{\mu,\nu}\\
=%
{\displaystyle\sum\limits_{j=0}^{s}}
a_{j}^{\mu,\nu}%
{\displaystyle\sum\limits_{k=j}^{s}}
\alpha_{s,k}^{\mu,\nu}\beta_{k,j}^{\mu,\nu}=%
{\displaystyle\sum\limits_{j=0}^{s}}
a_{j}^{\mu,\nu}\delta_{j,s}=a_{s}^{\mu,\nu}.
\end{gather*}
The same calculation holds for $b_{2k+1}$ and $c_{k}^{\mu,\nu},$ proving
(\ref{as}).
\end{proof}

\begin{remark}
Nielsen \cite{MR1508992}, defined Kapteyn series of the second type as%
\[
\left(  \frac{z}{2}\right)  ^{\frac{\mu+\nu}{2}}F\left(  z\right)  =%
{\displaystyle\sum\limits_{n=0}^{\infty}}
a_{n}^{\mu,\nu}\mathrm{J}_{\frac{\mu+n}{2}}\left[  \left(  \frac{\mu+\nu}%
{2}+n\right)  z\right]  \mathrm{J}_{\frac{\nu+n}{2}}\left[  \left(  \frac
{\mu+\nu}{2}+n\right)  z\right]
\]
and assuming that%
\[
F\left(  z\right)  =%
{\displaystyle\sum\limits_{m=0}^{\infty}}
B_{m}\left(  \frac{z}{2}\right)  ^{m},
\]
he obtained%
\begin{align*}
a_{s}^{\mu,\nu}  &  =%
{\displaystyle\sum\limits_{k=0}^{\left\lfloor \frac{s}{2}\right\rfloor }}
\left(  \frac{\mu+\nu}{2}+s-2k\right)  \Gamma\left(  \frac{\mu+s}%
{2}-k+1\right)  \Gamma\left(  \frac{\nu+s}{2}-k+1\right) \\
&  \times\binom{\frac{\mu+\nu}{2}+s-k-1}{k}\left(  \frac{2}{\mu+\nu
+2s}\right)  ^{\frac{\mu+\nu}{2}+s-k+1}B_{s-2k}.
\end{align*}

\end{remark}

We now have all the necessary tools to compute (\ref{S1}).

\begin{example}
Let
\[
a_{n}^{\mu,\nu}(p)=n^{2p},\quad c_{n}^{\mu,\nu}=0,\quad\mu=\nu=0.
\]
Then, $b_{2s+1}=0$ and we have%
\[
b_{2s}(p)=\frac{\left(  -1\right)  ^{s}}{\left(  s!\right)  ^{2}}%
{\displaystyle\sum\limits_{k=0}^{s}}
\left(  -1\right)  ^{k}\binom{2s}{s-k}k^{2\left(  s+p\right)  },
\]
or
\[
b_{2s}(p)=\frac{1}{\left(  s!\right)  ^{2}}%
{\displaystyle\sum\limits_{k=0}^{s}}
\left(  -1\right)  ^{k}\binom{2s}{k}\left(  s-k\right)  ^{2\left(  s+p\right)
}.
\]
Since the terms in the sum are symmetric with respect to $k=s,$ we can write%
\[
b_{2s}(p)=\frac{1}{2}\frac{1}{\left(  s!\right)  ^{2}}%
{\displaystyle\sum\limits_{k=0}^{2s}}
\left(  -1\right)  ^{k}\binom{2s}{k}\left(  s-k\right)  ^{2\left(  s+p\right)
},
\]
unless $s=0=p,$ in which case we have%
\[
b_{0}(0)=1.
\]
Thus, we have%
\begin{align*}
b_{2s}(p)  &  =\varepsilon_{s,p}\frac{1}{\left(  s!\right)  ^{2}}%
{\displaystyle\sum\limits_{k=0}^{2s}}
\left(  -1\right)  ^{k}\binom{2s}{k}\left(  s-k\right)  ^{2\left(  s+p\right)
}\\
&  =\varepsilon_{s,p}\frac{1}{\left(  s!\right)  ^{2}4^{s+p}}%
{\displaystyle\sum\limits_{k=0}^{2s}}
\left(  -1\right)  ^{k}\binom{2s}{k}\left(  2s-2k\right)  ^{2\left(
s+p\right)  },
\end{align*}
where $\varepsilon_{s,p}$ was defined in (\ref{epsilon}). Using
(\ref{binomial}), we get%
\[
b_{2s}(p)=\frac{\varepsilon_{s,p}}{4^{p}\left(  s!\right)  ^{2}}\left.
\frac{d^{2\left(  s+p\right)  }}{dt^{2\left(  s+p\right)  }}\ \sinh
^{2s}\left(  t\right)  \right\vert _{t=0},
\]
or
\begin{equation}
b_{2s}(p)=\frac{\varepsilon_{s,p}}{4^{p}\left(  s!\right)  ^{2}}\left(
2s+2p\right)  !\left[  t^{2p}\right]  \left[  \frac{\sinh\left(  t\right)
}{t}\right]  ^{2s}, \label{bs1}%
\end{equation}
and therefore%
\begin{align*}
b_{2s}(0)  &  =\varepsilon_{s,0}\frac{\left(  2s\right)  !}{\left(  s!\right)
^{2}},\quad b_{2s}(1)=\frac{1}{2}\frac{\left(  2s+2\right)  !}{4\left(
s!\right)  ^{2}}\frac{s}{3},\\
b_{2s}(2)  &  =\frac{1}{2}\frac{\left(  2s+4\right)  !}{4^{2}\left(
s!\right)  ^{2}}\frac{s\left(  5s-1\right)  }{90},\\
b_{2s}(3)  &  =\frac{1}{2}\frac{\left(  2s+6\right)  !}{4^{3}\left(
s!\right)  ^{2}}\frac{s\left(  35s^{2}-21s+4\right)  }{5670},\\
b_{2s}(4)  &  =\frac{1}{2}\frac{\left(  2s+8\right)  !}{4^{4}\left(
s!\right)  ^{2}}\frac{s\left(  5s-2\right)  \left(  35s^{2}-28s+9\right)
}{340200},\ldots.
\end{align*}

Let $g_{p}\left(  z\right)  $ be defined by
\[%
{\displaystyle\sum\limits_{n=0}^{\infty}}
n^{2p}\mathrm{J}_{n}^{2}\left(  2nz\right)  =g_{p}\left(  z\right)  =%
{\displaystyle\sum\limits_{s=0}^{\infty}}
b_{2s}(p)z^{2s}.
\]
Then, since
\[
\left(  p+1\right)  _{s}\left(  p+\frac{1}{2}\right)  _{s}=\frac{\left(
2s+2p\right)  !}{\left(  2p\right)  !4^{s}},
\]
we get from (\ref{bs1}) that%
\[
g_{p}\left(  z\right)  =\frac{\left(  2p\right)  !}{4^{p}}\left[
t^{2p}\right]
{\displaystyle\sum\limits_{s=0}^{\infty}}
\varepsilon_{s,p}\frac{\left(  p+1\right)  _{s}\left(  p+\frac{1}{2}\right)
_{s}}{\left(  s!\right)  ^{2}}\ \left[  2\frac{\sinh\left(  t\right)  }%
{t}z\right]  ^{2s}%
\]
\
\begin{equation}
=\frac{\left(  2p\right)  !}{4^{p}}\left[  t^{2p}\right]  \left\{  \frac{1}%
{2}+\frac{1}{2}\ _{2}F_{1}\left[  \left.
\begin{array}
[c]{c}%
p+1,p+\frac{1}{2}\\
1
\end{array}
\right\vert 4\frac{\sinh^{2}\left(  t\right)  }{t^{2}}z^{2}\right]  \right\}
. \label{gp}%
\end{equation}
Using (\ref{gp}), we obtain%
\begin{align*}
g_{0}\left(  z\right)   &  =\frac{1}{2}+\frac{1}{2\sqrt{1-4z^{2}}},\quad
g_{1}\left(  z\right)  =\frac{z^{2}\left(  1+z^{2}\right)  }{\left(
1-4z^{2}\right)  ^{\frac{7}{2}}},\\
g_{2}\left(  z\right)   &  =\frac{z^{2}\left(  1+37z^{2}+118z^{4}%
+27z^{6}\right)  }{\left(  1-4z^{2}\right)  ^{\frac{13}{2}}},\\
g_{3}\left(  z\right)   &  =\frac{z^{2}\left(  1+217z^{2}+5036z^{4}%
+23630z^{6}+22910z^{8}+2250z^{10}\right)  }{\left(  1-4z^{2}\right)
^{\frac{19}{2}}},\ldots.
\end{align*}

\end{example}

\begin{remark}
Formulas for $g_{0}\left(  z\right)  $ and $g_{1}\left(  z\right)  $ were
computed by George Augustus Schott in \cite{Schott}. They were reproduced by
Watson in \cite[17.6 (2)-(3)]{MR1349110}, but there is a typographical mistake
in the equation for $g_{1}\left(  z\right)  ,$ since the denominator should
contain a $\frac{7}{2}$ power, instead of the $\frac{1}{2}$ printed.
\end{remark}

\begin{acknowledgement}
The work of D. Dominici was supported by a Humboldt Research Fellowship for
Experienced Researchers from the Alexander von Humboldt Foundation.
\end{acknowledgement}

\bibliographystyle{elsarticle-num}

\begin{thebibliography}{00}

\bibitem{Bessel}
F.~W. Bessel.
\newblock Analytische {A}ufl\"{o}sung der {K}eplerschen {A}ufgabe.
\newblock {\em Abh. {P}reu\ss. Akad. Wiss. Berlin}, XXV:49--55, 1819.

\bibitem{Carlini}
F.~Carlini.
\newblock Ricerche sulla convergenza della serie che serve alla soluzione del
  problema di {K}eplero.
\newblock {\em Effem. Astron. (Milano)}, 44(Appendice):3--48, 1817.

\bibitem{Jacobi1850}
F.~Carlini.
\newblock Untersuchungen \"{u}ber die {C}onvergenz der {R}eihe durch welche das
  {K}epler'sche {P}roblem gel\"{o}st wird.
\newblock {\em Schumacher Astron. Nachr.}, 30(14):197--212, 1850.

\bibitem{MR1140277}
P.~Colwell.
\newblock Bessel functions and {K}epler's equation.
\newblock {\em Amer. Math. Monthly}, 99(1):45--48, 1992.

\bibitem{MR1268639}
P.~Colwell.
\newblock {\em Solving {K}epler's equation over three centuries}.
\newblock Willmann-Bell Inc., Richmond, VA, 1993.

\bibitem{MR2326079}
D.~Dominici.
\newblock A new {K}apteyn series.
\newblock {\em Integral Transforms Spec. Funct.}, 18(5-6):409--418, 2007.

\bibitem{MR1346357}
J.~Dutka.
\newblock On the early history of {B}essel functions.
\newblock {\em Arch. Hist. Exact Sci.}, 49(2):105--134, 1995.

\bibitem{MR698780}
A.~Erd{\'e}lyi, W.~Magnus, F.~Oberhettinger, and F.~G. Tricomi.
\newblock {\em Higher transcendental functions. {V}ol. {II}}.
\newblock Robert E. Krieger Publishing Co. Inc., Melbourne, Fla., 1981.

\bibitem{MR1915514}
N.~Fr{\"o}man and P.~O. Fr{\"o}man.
\newblock {\em Physical problems solved by the phase-integral method}.
\newblock Cambridge University Press, Cambridge, 2002.

\bibitem{Jacobi1849}
C.~G.~J. Jacobi.
\newblock Ueber die ann\"{a}hernde {B}estimmung sehr entfernter {G}lieder in
  der {E}ntwickelung der elliptischen {C}oordinaten nebst einer {A}usdehnung
  der {L}aplaceschen {M}ethode zur {B}estimmung der {F}unctionen gerader
  {Z}ahlen.
\newblock {\em Schumacher Astron. Nachr.}, 28(17):257--272, 1849.

\bibitem{MR1508887}
W.~Kapteyn.
\newblock Recherches sur les fonctions de {F}ourier-{B}essel.
\newblock {\em Ann. Sci. \'Ecole Norm. Sup. (3)}, 10:91--122, 1893.

\bibitem{Kapteyn1}
W.~Kapteyn.
\newblock On an expansion of an arbitrary function in a series of {B}essel
  functions.
\newblock {\em Messenger of Math.}, 35:122–--125, 1906.

\bibitem{Lagrange1768}
J.~L. Lagrange.
\newblock Nouvelle m\`{e}thode pour r\`{e}soudre les \`{e}quations
  litt\`{e}rales par le moyen des s\`{e}ries.
\newblock {\em M\`{e}m. de l'{A}cad. des {S}ci. {B}erlin}, XXIV:251--326, 1768.

\bibitem{Lagrange1771}
J.~L. Lagrange.
\newblock Sur le {P}robl\`{e}me de {K}\'{e}pler.
\newblock {\em M\`{e}m. de l'{A}cad. des {S}ci. {B}erlin}, XXV:204--233, 1771.

\bibitem{Tautz1}
I.~Lerche and R.~C. Tautz.
\newblock A note on summation of {K}apteyn series in astrophysical problems.
\newblock {\em Astrophys. J.}, 665(2):1288--1291, 2007.

\bibitem{MR2449307}
I.~Lerche and R.~C. Tautz.
\newblock Kapteyn series arising in radiation problems.
\newblock {\em J. Phys. A}, 41(3):035202, 10, 2008.

\bibitem{Tautz2}
I.~Lerche, R.~C. Tautz, and D.~S. Citrin.
\newblock Terahertz-sideband spectra involving {K}apteyn series.
\newblock {\em J. Phys. A}, 42(36):365206, 9, 2009.

\bibitem{Lommel}
E.~C. J.~v. Lommel.
\newblock {\em Studien \"{u}ber die {B}essel'schen {F}unktionen}.
\newblock Leipzig, 1868.

\bibitem{Meissel1}
E.~Meissel.
\newblock Neue {E}ntwicklungen \"{u}ber die {B}essel'schen {F}unctionen.
\newblock {\em Astron. Nachr.}, 129(3089):281--284, 1892.

\bibitem{Meissel2}
E.~Meissel.
\newblock Weitere {E}ntwicklungen \"{u}ber die {B}essel'schen {F}unctionen.
\newblock {\em Astron. Nachr.}, 130(3116):363--368, 1892.

\bibitem{MR1508992}
N.~Nielsen.
\newblock Recherches sur les s\'eries de fonctions cylindriques dues \`a {C}.
  {N}eumann et {W}. {K}apteyn.
\newblock {\em Ann. Sci. \'Ecole Norm. Sup. (3)}, 18:39--75, 1901.

\bibitem{Nielsen}
N.~Nielsen.
\newblock {\em Handbuch der {T}heorie der {Z}ylinderfunktionen}.
\newblock Druck und Verlag von B. G. Teubner, Leipzig, 1904.

\bibitem{Kapteyn}
E.~R. Paul.
\newblock The life and works of {J}. {C}. {K}apteyn.
\newblock {\em Space Science Reviews}, 64(1-2):93--174, 1993.

\bibitem{MR1344309}
J.~Peetre.
\newblock Outline of a scientific biography of {E}rnst {M}eissel (1826--1895).
\newblock {\em Historia Math.}, 22(2):154--178, 1995.

\bibitem{Schott}
G.~A. Schott.
\newblock {\em Electromagnetic {R}adiation}.
\newblock Cambridge University Press, Cambridge, 1912.

\bibitem{MR1349110}
G.~N. Watson.
\newblock {\em A treatise on the theory of {B}essel functions}.
\newblock Cambridge Mathematical Library. Cambridge University Press,
  Cambridge, 1995.

\end{thebibliography}

\end{document}